%% file: paper.tex
\documentclass[reqno]{amsart}
\numberwithin{equation}{section}

\usepackage[colorlinks,linkcolor={blue}]{hyperref}

\usepackage{bm}
\usepackage{amssymb}
\usepackage[alphabetic,nobysame]{amsrefs}
\usepackage{graphicx} 
\usepackage{amscd}
\usepackage{enumerate}
\usepackage{cancel}
\usepackage[font=footnotesize]{caption}
\usepackage{dsfont}
\usepackage[colorinlistoftodos]{todonotes}
\usepackage{tikz-cd}
\usepackage[all]{xy}
\usepackage{mathtools}
\usepackage{scrextend}

\definecolor{darkred}{rgb}{1,0,0} 
\definecolor{darkgreen}{rgb}{0,0.8,0}
\definecolor{darkblue}{rgb}{0,0,1}

\hypersetup{colorlinks,
linkcolor=darkblue,
filecolor=darkgreen,
urlcolor=darkred,
citecolor=darkgreen}

\makeatletter
\providecommand\@dotsep{5}
\makeatother

\definecolor{orange}{RGB}{253,85,0}
\definecolor{darkgreen}{RGB}{0,95,10}

\newcommand{\dgray}[1]{{\textcolor{darkgray}{#1}}}

 \newcommand{\Z}{\mathds{Z}}
 
 \newcommand{\R}{\mathds{R}}
 \newcommand{\Su}{\mathds{S}^1}

 \newcommand{\C}{\mathds{C}}

 \newcommand{\JJ}{J}

 \newcommand{\crit}{\mathrm{crit}}
 \newcommand{\reg}{\mathrm{reg}}
 
 \newcommand{\id}{\mathrm{id}}

 \newcommand{\vertiii}[1]{{\left\vert\kern-0.25ex\left\vert\kern-0.25ex\left
      \vert #1 \right\vert\kern-0.25ex\right\vert\kern-0.25ex\right\vert}}

 \DeclareMathOperator{\interior}{int}

 \DeclareMathOperator*{\ttoup}{\llongrightarrow}

\DeclareRobustCommand{\llongrightarrow}{\relbar\joinrel\relbar\joinrel\rightarrow}

 \theoremstyle{plain}
 \newtheorem{MainThm}{Theorem}

 \theoremstyle{definition} 
 
 \theoremstyle{plain}

 \newtheorem{Thm}{Theorem}[section]
 \newtheorem{Prop}[Thm]{Proposition}
 \newtheorem{Lemma}[Thm]{Lemma}
 \newtheorem{Cor}[Thm]{Corollary}
 
 \theoremstyle{definition}

 \newtheorem{Remark}[Thm]{Remark}
 \newtheorem{Example}[Thm]{Example}

\title{On potentials whose level sets are orbits}

\author[P. Bolle]{Philippe Bolle}
\address{Philippe Bolle\newline\indent LMA, Avignon Université\newline\indent 301 rue Baruch de Spinoza, BP 21239, F-84 916 Avignon, France}
\email{philippe.bolle@univ-avignon.fr}

\author[M. Mazzucchelli]{Marco Mazzucchelli}
\address{Marco Mazzucchelli\newline\indent CNRS, UMPA, \'Ecole Normale Sup\'erieure de Lyon\newline\indent 46 all\'ee d'Italie, 69364 Lyon, France}
\email{marco.mazzucchelli@ens-lyon.fr}

\author[A. Venturelli]{Andrea Venturelli}
\address{Andrea Venturelli\newline\indent LMA, Avignon Université\newline\indent 301 rue Baruch de Spinoza, BP 21239, F-84 916 Avignon, France}
\email{andrea.venturelli@univ-avignon.fr}

\date{March 28, 2024}



\thanks{M.~Mazzucchelli and A.~Venturelli are supported by the ANR grants CoSyDy (ANR-CE40-0014). M.~Mazzucchelli is also supported by the ANR grant COSY (ANR-21-CE40-0002).}

\begin{document}

\begin{abstract}
A level orbit of a mechanical Hamiltonian system is a solution of Newton equation that is contained in a level set of the potential energy. In 2003, Mark Levi asked for a characterization of the smooth potential energy functions on the plane with the property that any point on the plane lies on a level orbit; we call such functions Levi potentials. The basic examples are the radial monotone increasing smooth functions. In this paper we show that any Levi potential that is analytic or has totally path-disconnected critical set must be radial. Nevertheless, we show that every compact convex subset of the plane is the critical set of a Levi potential. A crucial observation for these theorems is that, outside the critical set, the family of level sets of a Levi potential forms a solution of the inverse curvature flow.

\tableofcontents
\end{abstract}

\maketitle

\vspace{-20pt}

\section{Introduction}

Let $U:\R^2\to\R$ be a smooth (meaning $C^\infty$) function, which will play the role of a potential energy, and consider a solution $q:I\to\R$ of the Newton equation $\ddot q=-\nabla U(q)$ defined on a maximal time interval $I\subset\R$. Such a solution is called a \emph{level orbit} when the function $s\mapsto U(q(s))$ is constant on $I$. In 2003, Mark Levi \cite{Levi:2003aa} asked for a  characterization of the smooth potentials $U:\R^2\to\R$ with the property that any point of $\R^2$ lies on a level orbit. We refer to the functions $U$ satisfying this property as to \emph{Levi potentials}.
The basic examples are the radial functions 
$U:\R^2\to\R$, $U(q)=u(\|q-q_0\|^2)$, 
where $u:[0,+\infty)\to\R$ is a monotone increasing smooth function. Indeed, any
  point $q_1\in\R^2$ lies on the level orbit $q(t)=e^{\omega t J}(q_1-q_0)+q_0$, where $\omega^2=2\dot u(\|q-q_0\|^2)$ and $J:\R^2\to\R^2$, $J(x,y)=(-y,x)$ is the complex structure of $\R^2$.

Clearly, one can easily construct  a non-radial Levi potential $V$ by patching together radial ones: given a constant $c\in\R$, a finite family of pairwise disjoint open balls $B_i=B^2(q_i,r_i)\subset\R^2$, and monotone increasing smooth functions $u_i:[0,+\infty)\to\R$ such that $u_i|_{[r_i,+\infty)}\equiv c$, a Levi potential is given by
\begin{align*}
V(q)
=
\left\{
  \begin{array}{@{}ll}
    u_i(\|q-q_i\|^2), & q\in B_i,\vspace{5pt} \\ 
    c, & q\not\in\bigcup_i B_i. 
  \end{array}
\right.
\end{align*}
Notice that the critical set of $V$ contains the boundaries of all the balls $B_i$, and in particular is not totally path-disconnected. Our first theorem asserts that this is indeed necessary for an exotic (meaning non-radial) Levi potential.

\begin{MainThm}\label{t:smooth}
Any Levi potential $U:\R^2\to\R$ whose critical set $\crit(U)$ is totally path-disconnected has a unique critical point and is radial.
\end{MainThm}

In the real analytic category, the situation is completely rigid, as asserted by our second theorem.

\begin{MainThm}\label{t:analytic}
Any analytic Levi potential is radial.
\end{MainThm}

Nevertheless, more exotic (non-analytic) Levi potentials, not obtained by just patching radial ones together, do exist. 

\begin{MainThm} \label{t:nonradial}
For any non-empty compact convex subset $C\subset\R^2$ there exist a Levi potential $U:\R^2\to\R$ with critical set $\crit(U)=C$.
\end{MainThm}

The crucial observation behind the proofs of these theorems is the link between Levi potentials and the solutions of the inverse curvature flow in the Euclidean plane, which is the geometric evolution PDE $\partial_t\gamma_t(s)=K_{\gamma_t}(s)^{-1}N_{\gamma_t}(s)$. Here $\gamma_t$ is a family of smooth curves on the plane with negative normal vector field $N_{\gamma_t}$ and signed geodesic curvature $K_{\gamma_t}$. As it turns out, outside the critical set of a Levi potential $U$, the family of level curves $U^{-1}(c)$, $c\in\R$, can be parametrized into a solution of the inverse curvature flow (Proposition~\ref{p:level_orbit_inv_curv_flow}).

There is a large literature on geometric flows for families of curves evolving according to a function $f$ of the curvature, such as the curve shortening flow \cite{Gage:1983aa, Gage:1984aa, Grayson:1989aa} for $f=\id$, and flows allowing general monotone decreasing $f$ \cite{Chow:1996aa}. Generalization of the inverse curvature flow in higher dimension has also been widely studied in the literature since the work of Gerhardt \cite{Gerhardt:1990aa}, Urbas \cite{Urbas:1991aa}, and Huisken-Ilmanen \cite{Huisken:2001aa, Huisken:2008aa}. For the proofs of Theorems~\ref{t:smooth} and~\ref{t:analytic}, we will need a recent result of Risa and Sinestrari \cite{Risa:2020aa}, which in low dimension asserts that the unique curves admitting solutions of the inverse curvature flow for all negative times are the round circles (Proposition~\ref{p:periodic}). We will also need a  result on the non-existence of solutions of the inverse curvature flow starting on properly embedded open curves (Proposition~\ref{p:open}).

While the present work focuses on mechanical Hamiltonian systems on the plane, one could generalize the notion of Levi potential to higher dimensions in several possible ways. Such generalizations, and rigidity/flexibility results in the spirit of Theorems~\ref{t:smooth}, \ref{t:analytic}, and \ref{t:nonradial}, are the subject of ongoing investigation.

\subsection{Organization of the paper}
In Section~\ref{s:inverse_curvature_flow}, we recall some known features of the inverse curvature flow, and establish a few new ones. In particular, we provide a short proof of the low dimensional version of Risa and Sinestrari's theorem (Proposition~\ref{p:periodic}), the proof of the non-existence of solutions of the inverse curvature flow starting on properly embedded open curves (Proposition~\ref{p:open}), and the proof of the existence of solutions starting on curves that are the boundary of a given convex compact set (Lemmas~\ref{l:PDE_support} and \ref{l:PDE_support_2}). In Section~\ref{s:Levi_potentials}, we establish several properties of Levi potentials, in particular drawing the connection with the inverse curvature flow, and prove Theorems~\ref{t:smooth}, \ref{t:analytic}, and \ref{t:nonradial}.

\subsection{Acknowledgments}
We thank Mark Levi for raising the original question that motivated the present work. We also thank the anonymous referee for insightful comments, and for pointing out several references.

\section{The inverse curvature flow}\label{s:inverse_curvature_flow}

\subsection{The PDE}\label{ss:PDE}
Let $\JJ:\R^2\to\R^2$, $\JJ(x,y)=(-y,x)$ be the complex structure of $\R^2$.
For a smooth immersed curve $\gamma:\R\looparrowright\R^2$, we denote by $T_\gamma:\R\to\R^2$ its positive unit tangent vector field, by $N_\gamma:\R\to\R^2$ its negative normal vector field, and by $K_{\gamma}:\R\to\R^2$ its signed geodesic curvature, i.e.
\begin{align*}
T_\gamma(s)=\frac{\dot\gamma(s)}{\|\dot\gamma(s)\|},
\quad
N_\gamma(s)=-\JJ T_\gamma(s),
\quad
K_\gamma(s)=-\frac{\langle\dot T_\gamma(s),N_\gamma(s)\rangle}{\|\dot\gamma(s)\|}.
\end{align*}
We consider the PDE
\begin{align}\label{e:PDE}
\partial_t\gamma_t=\frac{N_{\gamma_t}}{K_{\gamma_t}},
\end{align}
where each $\gamma_t$ is a smooth immersed curve as above with nowhere-vanishing geodesic curvature. If \eqref{e:PDE} admits a family of solutions $\gamma_t$ for $t\in(t_0,t_1)$, we say that such family evolves according to the inverse curvature flow.

\begin{Example}\label{ex:circular_solution}
The simplest example of solution of the inverse curvature flow is the \emph{circular} one, given by
\begin{align*}
 \gamma_t(s):=q_0 + e^{t} e^{v s \JJ}q_1,\qquad\forall t,s\in\R,
\end{align*}
where $q_0\in\R^2$, $q_1\in\R^2\setminus\{0\}$, and $v\in\R\setminus\{0\}$. Every $\gamma_t$ is a $2\pi|v|^{-1}$-periodic parametrization of a circle of radius $e^t\|q_1\|$ centered at $q_0$.
\end{Example}

Notice that the PDE~\eqref{e:PDE} is geometric: if $\gamma_t(s)$ is a solution defined for all $t\in(t_0,t_1)$ and $s\in(s_0,s_1)$, the composition of the family of curves $\gamma_t$ with a diffeomorphism of the form $(s_0',s_1')\to(s_0,s_1)$ is still a solution.
The following lemma describes the evolution of the speed of the solutions.

\begin{Lemma}\label{l:evolution_speed}
Let $\gamma_t$ be a solution of~\eqref{e:PDE}. Then $\|\dot\gamma_{t_1}(s)\|=e^{t_1-t_0}\|\dot\gamma_{t_0}(s)\|$ for all $s,t_0,t_1$ for which both sides are defined.
\end{Lemma}

\begin{proof}
Since $\dot N_{\gamma_t}=K_{\gamma_t}\dot\gamma_t$, we have
\begin{align*}
\partial_t (\|\dot\gamma_t(s)\|^2)
& =
2 \langle \dot\gamma_t(s),\partial_t\dot\gamma_t(s)\rangle
=
2\langle \dot{\gamma_t}(s) , \partial_s \big(\tfrac{N_{\gamma_t}(s)}{K_{\gamma_t}(s)}\big)\rangle \\
& = 
2\langle \dot{\gamma_t}(s) , \tfrac{\dot N_{\gamma_t}(s)}{K_{\gamma_t}(s)}\rangle 
= 
2\|\dot\gamma_t(s)\|^2.
\end{align*}
Therefore, for any fixed $s$, the function $t\mapsto \|\dot\gamma_t(s)\|^2$ is a solution of the ODE $\partial_t (\|\dot\gamma_t(s)\|^2)=2\|\dot\gamma_t(s)\|^2$, and our claim follows.
\end{proof}

Lemma~\ref{l:evolution_speed} implies that each curve $\gamma_t$ in a family evolving under the inverse curvature flow is parametrized with constant speed provided the same is true for one value of $t$. Therefore, from now on, we will always assume that 
$\partial_s \|\dot\gamma_t(s)\|\equiv0$. Notice that, if $\gamma_0$ is a periodic curve, all the $\gamma_t$ are periodic curves of the same period; up to a time-reparametrization independent of $t$, all such curves $\gamma_t$ can be defined on the circle $\Su=\R/2\pi\Z$.

The following statement, due to Risa-Sinestrari \cite{Risa:2020aa}, is one of the ingredients for Theorem \ref{t:smooth} and \ref{t:analytic}. Actually, their result holds in arbitrary  dimension for certain hypersurfaces evolving with the inverse mean-curvature flow. For the reader convenience, we shall provide a simple proof of the stated low dimensional result in Section~\ref{ss:proofs_propositions}.

\begin{Prop}\label{p:periodic}
The only periodic solutions $\gamma_t:\Su\to\R^2$ of the inverse curvature flow defined for all $t\leq0$ are the circular ones $($Example~\ref{ex:circular_solution}$)$.
\end{Prop}

We will also need the following non-existence result,  proved in Section~\ref{ss:proofs_propositions}. The statement also follows from \cite[Theorem~1.3]{Choi:2021aa}.

\begin{Prop}\label{p:open}
 There is no solution of the inverse curvature flow $\gamma_t$, defined for $t\in[0,\epsilon]$ with $\epsilon>0$, such that $\gamma_0:\R\hookrightarrow\R^2$ is a proper embedding.
\end{Prop}

Proposition \ref{p:open} does not hold in higher dimension: Daskalopoulos-Huisken \cite{Daskalopoulos:2022aa} proved a global existence result for the inverse mean curvature flow in $\R^{n+1}$, with $n\ge 2$, starting from a strictly mean convex starshaped entire graph with superlinear growth at infinity.

\subsection{Rescaled solutions}\label{ss:rescaled_PDE}
Let $\gamma_t$ be a solution of the inverse curvature flow defined for all $t$ in some interval containing 0. 
 We assume without loss of generality that each $\gamma_t$ is parametrized with constant speed. Clearly, the image of $\gamma_t$ under a homothety is still a solution of the inverse curvature flow. Therefore, up to rescaling $\gamma_0$ by scalar multiplication with $\|\dot\gamma_0\|^{-1}$, we can assume without loss of generality that $\|\dot\gamma_0\|\equiv1$, and Lemma~\ref{l:evolution_speed} implies
$\|\dot\gamma_t\|\equiv e^t$.
We define the rescaled curves
$\sigma_t:=e^{-t}\gamma_t$, which are parametrized with unit speed $\|\dot\sigma_t\|\equiv1$.

\begin{Lemma}
The curves $\sigma_t$ satisfy the PDE
\begin{align*}
\partial_t\sigma_t=-\sigma_t+\frac{N_{\sigma_t}}{K_{\sigma_t}}.
\end{align*}
The associated functions 
$a_t(s):=K_{\sigma_t}(s)^{-1}$
satisfy the PDE
\begin{align*}
\partial_t a_t=a_t^2 \ddot a_t.
\end{align*}
\end{Lemma}

\begin{proof}
Notice that $\dot\sigma_t=T_{\gamma_t}$, $N_{\sigma_t}=N_{\gamma_t}$, and  $K_{\sigma_t}=e^t K_{\gamma_t}$. This, together with the usual relations $\ddot\sigma_t=-K_{\sigma_t}N_{\sigma_t}$ and $\dot N_{\sigma_t}=K_{\sigma_t}\dot\sigma_t$, implies
\begin{align*}
 \partial_t\sigma_t 
 = 
 - e^{-t}\gamma_t +  \frac{N_{\gamma_t}}{e^t K_{\gamma_t}}
 =
 - \sigma_t + \frac{N_{\sigma_t}}{K_{\sigma_t}}
 .
\end{align*}
We compute
\begin{align*}
 \partial_s^2\partial_t\sigma_t
 &=
 \partial_s^2 \bigg(- \sigma_t + \frac{N_{\sigma_t}}{K_{\sigma_t}}\bigg)
 =
 \partial_s \bigg(\underbrace{- \dot\sigma_t + \frac{\dot N_{\sigma_t}}{K_{\sigma_t}}}_{=0} + \dot a_t N_{\sigma_t}\bigg)\\
 &= \ddot a_t N_{\sigma_t} + \dot a_t \dot N_{\sigma_t} 
 = \ddot a_t N_{\sigma_t} + \frac{\dot a_t}{a_t} \dot\sigma_t .
\end{align*}
Switching the order of the derivatives on the left-hand side, we obtain
\begin{align*}
 \partial_t\ddot\sigma_t
 &=
 -\partial_t (K_{\sigma_t}N_{\sigma_t})
 =
 -(\partial_t K_{\sigma_t})N_{\sigma_t} - K_{\sigma_t}(\partial_t N_{\sigma_t})\\
 & =
 \frac{\partial_t a_t}{a_t^2}N_{\sigma_t} - K_{\sigma_t}(\partial_t N_{\sigma_t}).
\end{align*}
Notice that $\langle\partial_tN_{\sigma_t},N_{\sigma_t}\rangle=\tfrac12 \partial_t \|N_{\sigma_t}\|^2 =0$. Therefore,
\begin{align*}
 \ddot a_t & = \langle \partial_s^2\partial_t\sigma_t , N_{\sigma_t} \rangle
 = \langle \partial_t\ddot\sigma_t , N_{\sigma_t} \rangle
 =
 \frac{\partial_t a_t}{a_t^2}.
 \qedhere
\end{align*}
\end{proof}

\begin{Lemma}\label{l:circular_a}
If $\gamma_t$ is periodic, then it is a circular solution of the inverse curvature flow $($Example~\ref{ex:circular_solution}$)$ if and only if $\dot a_t\equiv 0$ for all values of $t$.
\end{Lemma}

\begin{proof}
If $\gamma_t$ is a circular solution of the inverse curvature flow, for each value of $t$ the curvature $K_{\gamma_t}$ is constant, and therefore $\dot a_t\equiv 0$.
Conversely, if $\dot a_t\equiv0$ for all values of $t$ for which $\gamma_t$ is defined, then each $\gamma_t$ is a curve of constant curvature. Namely, each $\gamma_t$ is a $2\pi|v|^{-1}$-periodic parametrization with constant speed of a round circle of radius $e^tr>0$, and therefore has the form
\begin{align*}
 \gamma_t(s)=q_t + e^{t} e^{v s \JJ} p,
\end{align*}
for some $q_t\in\R^2$, $p\in\R^2$ with $\|p\|=r>0$, and $v\in\R\setminus\{0\}$. The inverse curvature flow PDE \eqref{e:PDE} can be rewritten as
\begin{align*}
 \partial_t q_t + e^{t} e^{v s \JJ} p = e^{t} e^{v s \JJ} p,
\end{align*}
which implies $\partial_t q_t\equiv0$. Therefore $q_t=q_0$ for all $t$, and $\gamma_t$ is a circular solution of the inverse curvature flow.
\end{proof}

\subsection{Proofs of the propositions}\label{ss:proofs_propositions}
In this subsection we carry out the proofs of the statements given in Section~\ref{ss:PDE}.

\begin{proof}[Proof of Proposition~\ref{p:periodic}]
Let $\gamma_t:\Su\looparrowright\R^2$ be a periodic solution of the inverse curvature flow defined for all $t\leq0$. Up to rescaling and reparametrization, we can assume that $K_{\gamma_0}>0$ and $\|\dot\gamma_0\|\equiv1$, so that we are in the setting of Section~\ref{ss:rescaled_PDE}. We wish to prove that that $\gamma_t$ is a circular solution of the inverse curvature flow, or equivalently that the function $a_t:\Su\to(0,+\infty)$ is constant for all $t\leq0$ (Lemma~\ref{l:circular_a}). We proceed by contradiction, assuming that
\begin{align}
\label{e:contradiction_periodic}
 \dot a_0\not\equiv0.
\end{align}
Since
\begin{align*}
 \frac{d}{dt} \int_{\Su} \frac1{a_t(s)} ds
 =
 -
 \int_{\Su} \frac{\partial_t a_t(s)}{a_t(s)^2} ds
 =
 -
 \int_{\Su} \ddot a_t(s)\,ds
 =
 0,
\end{align*}
we have a constant value
\begin{align*}
 c_0:=\int_{\Su} \frac1{a_t(s)} ds\leq \frac{2\pi}{\displaystyle\min_{s\in\Su}a_t(s)},\qquad\forall t\leq0.
\end{align*}
In particular
\begin{align}
\label{e:c_1}
c_1:=\sup_{t\leq0} \min_{s\in\Su}a_t(s)<+\infty.
\end{align}
We define the smooth function 
\begin{align*}
 b:(-\infty,0]\to\R,
 \qquad
 b(t):=\int_{\Su} \ln(a_t(s))\,ds.
\end{align*}
Its derivative is non-positive, for
\begin{align*}
\dot b(t) 
= 
\int_{\Su} \frac{\partial_ta_t(s)}{a_t(s)}\,ds
= 
\int_{\Su} a_t(s)\,\ddot a_t(s)\,ds
=
-\int_{\Su} \dot a_t(s)^2ds
\leq0.
\end{align*}
By~\eqref{e:contradiction_periodic}, we have $\dot b(0)<0$. Moreover
\begin{align*}
\ddot b(t)
&=
-\int_{\Su} 2\,\dot a_t(s)\, \partial_t\dot a_t(s)\,ds
=
-\int_{\Su} 2\,\dot a_t(s)\, \partial_s \big(a_t(s)^2\ddot a_t(s)\big)\,ds \\
&=
\int_{\Su} 2\,\ddot a_t(s)^2 a_t(s)^2\,ds
\geq
0.
\end{align*}
Therefore $b$ is a convex function with negative derivative at the origin, and therefore 
\begin{align}
\label{e:b_diverges}
 \lim_{t\to-\infty} b(t)=+\infty.
\end{align}
The inequality 
\begin{align*}
 b(t)\leq2\pi \ln\Big(\max_{s\in\Su} a_t(s)\Big),
\end{align*}
together with~\eqref{e:c_1}, implies that 
\begin{align*}
 \max_{s\in\Su} a_t(s)
 -
 \min_{s\in\Su} a_t(s)
 \geq
 e^{b(t)/2\pi} - c_1.
\end{align*}
By~\eqref{e:b_diverges}, there exists $t_0<0$ such that 
\begin{align*}
 \max_{s\in\Su} a_t(s)
 -
 \min_{s\in\Su} a_t(s)
 \geq
 \frac{e^{b(t)/2\pi}}2,
 \qquad\forall t\leq t_0.
\end{align*}
This, together with the opposite estimate
\begin{align*}
 \max_{s\in\Su} a_t(s)
 -
 \min_{s\in\Su} a_t(s)
 \leq
 \int_{\Su} |\dot a_t(s)|\,ds
 \leq
 \sqrt{2\pi} 
 \left(\int_{\Su} \dot a_t(s)^2\,ds\right)^{1/2}
 =\sqrt{-2\pi\dot b(t)},
\end{align*}
implies that
\begin{align*}
 \dot b(t) \leq -\frac{e^{b(t)/\pi}}{8\pi},
 \qquad\forall t\leq t_0.
\end{align*}
This differential inequality cannot be satisfied for large negative values of $t$. Indeed, the positive function $f(t):=e^{-b(t)/\pi}$ has derivative 
\[\dot f(t)=-\frac{e^{-b(t)/\pi}\dot b(t)}\pi\geq\frac{1}{8\pi^2},\qquad\forall t\leq t_0.\]
But this implies that $f(t)<0$ for large negative values of $t$, which gives a contradiction.
\end{proof}

The proof of Proposition~\ref{p:open} requires the following elementary property of embedded open curves in the plane.

\begin{Lemma}\label{l:curvature_plane_lines}
Let $\gamma:\R\hookrightarrow\R^2$ be an embedding parametrized with unit speed $\|\dot\gamma\|\equiv1$ and with positive curvature $K_\gamma>0$, and consider the integral curvature
\begin{align*}
\Delta_\gamma:=\int_{-\infty}^{+\infty} K_\gamma(s)\,ds.
\end{align*}
Then one of the following two points holds.
\begin{itemize}

\item[(i)] $\Delta_\gamma\leq\pi$, and the embedding $\gamma:\R\hookrightarrow\R^2$ is proper.

\item[(ii)] $\Delta_\gamma=+\infty$, and at least one of the half curves $\gamma|_{(-\infty,0]}$ or $\gamma|_{[0,+\infty)}$ is bounded.

\end{itemize}
\end{Lemma}

\begin{proof}
We identify $\R^2$ with the complex plane $\C$, and  write the velocity vectors as $\dot\gamma(s)=e^{i\theta(s)}$ for some smooth functions $\theta:\R\to\R$, so that $K_{\gamma}(s)=\dot\theta(s)$. Since $K_{\gamma}$ is positive, we have limits 
\begin{align*}
 \theta_\pm:=\lim_{s\to\pm\infty}\theta(s),
\end{align*}
and $\Delta_\gamma=\theta_+-\theta_-$.
If $\theta_+$ is finite, we have
\begin{align*}
 \lim_{s\to+\infty} \frac{\|\gamma(s)-se^{i\theta_+}\|}s=0,
\end{align*} 
and in particular  $\|\gamma(s)\|$ diverges as $s\to+\infty$. 
Analogously, if $\theta_-$ is finite, then $\|\gamma(s)\|$ diverges as $s\to-\infty$. 
If $\Delta_\gamma\leq\pi$, then both $\theta_+$ and $\theta_-$ are finite, and therefore $\gamma:\R\hookrightarrow\R^2$ is proper.

Assume now that $\Delta_\gamma>\pi$, so that we can find $s_1<s_2$ such that $\theta(s_2)-\theta(s_1)\in(\pi,2\pi)$. The lines $\ell_1:=\gamma(s_1)+\R\dot\gamma(s_1)$ and $\ell_2:=\gamma(s_2)+\R\dot\gamma(s_2)$ are not parallel, and therefore intersect. 
Consider the compact set $C\subset\R^2$ bounded by $\gamma([s_1,s_2])$ together with portions of the lines $\ell_1$ and $\ell_2$ (Figure~\ref{f:incident}(a)). If $\gamma((-\infty,s_1])\subset C$ then $\theta_-=-\infty$; if instead $\gamma((-\infty,s_1])$ exits $C$, the positivity of the curvature $K_\gamma>0$ implies that there exists $s_0<s_1$ such that $\gamma([s_0,s_1])\subset C$ and $\gamma(s_0)\in\ell_2$ (Figure~\ref{f:incident}(b)). Analogous properties hold for the other end of the curve: if $\gamma([s_2,+\infty))\subset C$ then $\theta_+=+\infty$; if instead $\gamma([s_2,+\infty))$ exits $C$, then there exists $s_3>s_2$ such that $\gamma([s_2,s_3])\subset C$ and $\gamma(s_3)\in\ell_1$. This readily implies that $\gamma((-\infty,s_1])$ and $\gamma([s_2,+\infty))$ cannot both exit $C$, for otherwise $\gamma$ would have a self-intersection. Hence at least one of the inclusions $\gamma((-\infty,s_1])\subset C$ and $\gamma([s_2,+\infty))\subset C$ must hold, and $\Delta_\gamma=+\infty$.
\end{proof}

\begin{figure}
\begin{scriptsize}
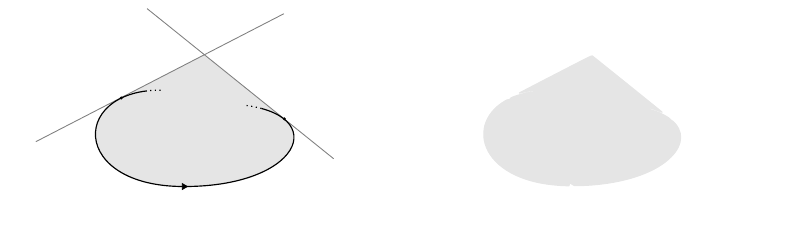
\end{scriptsize}
\caption{}
\label{f:incident}
\end{figure}

\begin{proof}[Proof of Proposition~\ref{p:open}]
Assume by contradiction that there exists a family of  curves $\gamma_t$, for $t\in[0,\epsilon]$, evolving according to the inverse curvature flow, such that $\gamma_{0}:\R\hookrightarrow\R^2$ is a proper embedding. Up to rescaling and reparametrization, we can assume that each $\gamma_t$ has positive curvature $K_{\gamma_t}>0$ and that $\gamma_0$ has unit speed $\|\dot\gamma_0\|\equiv1$, so that once again we are in the setting of Section~\ref{ss:rescaled_PDE}. Notice that the associated curves $\sigma_t=e^{-t}\gamma_t$ have also  positive curvature $K_{\sigma_t}>0$. By Lemma~\ref{l:curvature_plane_lines}, we have 
\begin{align}
\label{e:integral_curvature_bound}
 \int_{-\infty}^{+\infty} K_{\sigma_0}(s)\leq\pi.
\end{align}
In order to derive a contradiction, let us consider the inverse curvature function $a_t(s)=K_{\sigma_t}(s)^{-1}$, as in Section~\ref{ss:rescaled_PDE}, and define, for $T \in [0,\epsilon]$,
\begin{align*}
 A_T(s):=\int_0^T a_t(s)\,dt.
\end{align*}
Let us compute the derivative 
\begin{align}  \label{derAT}
 \dot A_T(S)
 &
 =
 \dot A_T(0) + \int_0^S \ddot A_T(s)\,ds
 =
 \dot A_T(0) + \int_0^S \!\!\int_0^T \ddot a_t(s)\,dt\,ds  \nonumber \\
 &=
 \dot A_T(0) + \int_0^S\!\!\int_0^T \frac{\partial_ta_t(s)}{a_t(s)^2}\,dt\,ds
 =
 \dot A_T(0) - \int_0^S\!\!\int_0^T \partial_t (a_t(s)^{-1})\,dt\,ds  \nonumber \\
 &=
 \dot A_T(0) - \int_0^S \big( a_T(s)^{-1}-a_0(s)^{-1} \big) \,ds  \nonumber \\
&=
 \dot A_T(0) + \int_0^S K_{\sigma_0}(s)\,ds - \int_0^S K_{\sigma_T}(s)\,ds.
\end{align}
Hence
\[
 \dot A_T(S)- \dot A_T(-S)=\int_{-S}^S K_{\sigma_0}(s)\,ds-
  \int_{-S}^S K_{\sigma_T}(s)\,ds.
\]
We infer 
\[
\liminf_{S \to +\infty} \int_{-S}^S K_{\sigma_T}(s)\,ds =
\liminf_{S \to +\infty} \Big( \int_{-S}^S K_{\sigma_0}(s)\,ds+\dot A_T(-S)  - \dot A_T(S)\Big) 
 \leq \pi .
\]
The latter inequality is a consequence of  \eqref{e:integral_curvature_bound} and
 of the fact that,  by the positivity of the map $S \mapsto A_T(S)+A_T(-S)$, 
 $\liminf_{S \to +\infty}  - \dot A_T(S)+ \dot A_T(-S) \leq 0$. 
Hence
\begin{equation}  \label{ICbound2}
\int_{-\infty}^{+\infty} K_{\sigma_T}(s)\,ds \leq \pi,
\qquad
\forall T \in [0,\epsilon].
\end{equation}
We  claim that $1/A_\epsilon$ is an $L^1$-function, i.e.
\begin{align}
\label{e:integrability_A}
 \int_{-\infty}^{+\infty} \frac1{A_\epsilon (s)}\,ds<+\infty .
\end{align}
Indeed, we have
\begin{align*}
\epsilon
&=
\int_0^\epsilon 
\frac{\sqrt{a_t(s)}}{\sqrt{a_t(s)}} \,dt 
\leq 
\bigg(
\int_0^\epsilon a_t(s)\,dt 
\bigg)^{1/2}
\bigg( 
\int_0^\epsilon \frac1{a_t(s)}\,dt 
\bigg)^{1/2}\\
&=
\sqrt{A_\epsilon(s)}
\bigg( 
\int_0^\epsilon \frac1{a_t(s)}\,dt 
\bigg)^{1/2},
\end{align*}
and therefore, by~\eqref{ICbound2},
\begin{align*}
 \int_{-\infty}^{+\infty} \frac1{A_\epsilon(s)}\,ds
 \leq
 \frac1{\epsilon^2}
 \int_0^\epsilon \!\!\int_{-\infty}^{+\infty} \frac1{a_t(s)}\,ds\,dt
 \leq
 \frac\pi \epsilon.
\end{align*}

Now \eqref{derAT}  together with~\eqref{ICbound2} and the fact 
that $K_{\sigma_T} (s) >0$, implies that 
\[\sup_{s\in\R} |\dot A_\epsilon(s)| \le|\dot A_\epsilon(0)|+\pi.\] 
Therefore, we have a linear bound 
\[A_\epsilon (s)\leq A_\epsilon (0) +   (|\dot A_\epsilon (0)|+\pi)|s|,\qquad\forall s\in \R,\]
which contradicts the integrability~\eqref{e:integrability_A}.
\end{proof}

\subsection{Evolution of the support functions}\label{ss:PDE_support}
Let $\gamma:\Su\hookrightarrow\R^2$ be an embedded smooth curve with positive curvature, which parametrizes the smooth boundary of a convex compact subset $C_\gamma\subset\R^2$. Let us recall the classical notion of support function of a convex body: for $C_\gamma\subset\R^2$, it is the smooth function
\begin{align*}
h_\gamma:\Su\to\R,
\qquad
h_\gamma(s)=\max_{q\in C_\gamma}\, \langle q,u(s)\rangle,
\end{align*}
where $u(s)=(\cos(s),\sin(s))$,
and it satisfies $h_\gamma+\ddot h_\gamma>0$. Conversely, any smooth function $h:\Su\to\R$ such that $h+\ddot h>0$ is the support function $h=h_\gamma$ of a unique (up to reparametrization) embedded smooth curve $\gamma:\Su\hookrightarrow\R^2$ with positive curvature, given by
\begin{align*}
 \gamma(s)=h_\gamma(s)u(s)+\dot h_\gamma(s)J u(s).
\end{align*}
The velocity vector of $\gamma$ can be expressed as
\begin{align*}
 \dot\gamma=\dot h_\gamma u+ h_\gamma J u +\ddot h_\gamma Ju + \dot h_\gamma JJu = (h_\gamma +\ddot h_\gamma )J u,
\end{align*}
and its curvature as
\begin{align}
\label{e:curvature_formula}
 K_\gamma(s)=
 \frac{\langle\dot\gamma,\dot N_{\gamma}\rangle}{\|\dot\gamma\|^2}
 =
 \frac{1}{\|\dot\gamma\|}
 =
 \frac{1}{h_\gamma +\ddot h_\gamma}.
\end{align}

The family of support functions of periodic curves evolving with the inverse curvature flow is described by the following well known statement, which is a particular case of a result of Tso \cite{Tso:1985aa,Urbas:1991aa,Chow:1996aa}. We provide the short proof here for the reader's convenience.

\begin{Lemma}\label{l:PDE_support}
Let $I$ be an interval with non-empty interior. 
A smooth family of functions $h_t:\Su\to\R^2$, $t\in I$, satisfies
\begin{align}
\label{e:PDE_support}
 \partial_t h_t = h_t+\ddot h_t>0
\end{align}
if and only if $h_t=h_{\gamma_t}$ for a suitably parametrized smooth family of embedded periodic curves with positive curvature $\gamma_t:\Su\hookrightarrow\R^2$ that satisfies the inverse curvature PDE \eqref{e:PDE}. 
\end{Lemma}

\begin{proof}
Let $\gamma_t:\Su\hookrightarrow\R^2$, $t\in I$, be a family of smooth embedded periodic curves of positive curvature evolving with the inverse curvature flow \eqref{e:PDE}. There is a unique smooth family of diffeomorphisms $\theta_t:\Su\to\Su$ such that the reparametrized curves $\zeta_t:=\gamma_t\circ \theta_t$ have negative normal vector field $N_{\zeta_t}(s)=u(s)$, and thus velocity vector $\dot\zeta_t=\|\dot\zeta_t\|Ju$. The associated support functions $h_t:\Su\to\R$, given by $h_t(s)=\langle u(s),\zeta_t(s) \rangle$, allow to write the curves as 
$\zeta_t=h_{t}u+\dot h_{t}J u$. Since the family of curves $\zeta_t : \Su\hookrightarrow\R^2$ satisfies \eqref{e:PDE}, 
by~\eqref{e:curvature_formula} we infer
\begin{align*}
 \partial_t h_t=\langle u,\partial_t\zeta_t \rangle = \frac{1}{K_{\zeta_t}} = \|\dot\zeta_t\| = h_t+\ddot h_t,
\end{align*}
and the last equality also implies that these terms are positive.

Conversely, assume that $h_t:\Su\to\R$, $t\in I$, satisfies \eqref{e:PDE_support}. Since $\ddot h_t+h_t>0$, the associated family of curves 
$\zeta_t:\Su\hookrightarrow\R^2$, $\zeta_t(s)=h_t(s)u(s)+\dot h_t(s)J u(s)$ has positive curvature $K_{\zeta_t}=(\ddot h_t+h_t)^{-1}=\|\dot\zeta_t\|^{-1}$, and we have
\begin{align*}
\partial_t\zeta_t
= 
\partial_t h_t\, u + \partial_t\dot h_t\, Ju
= 
(\ddot h_t + h_t)N_{\zeta_t} + \partial_t\langle Ju,\zeta_t\rangle J u
=
\frac{N_{\zeta_t}}{K_{\zeta_t}} + \underbrace{\langle Ju,\partial_t\zeta_t\rangle}_{=:f_t} T_{\zeta_t}.
\end{align*}
Let $\theta_t:\Su\to\Su$  be the smooth family of diffeomorphisms defined by $\theta_{t_0}=\id$ for an arbitrarily chosen $t_0\in I$, and satisfying the ODE
\begin{align*}
 \partial_t \theta_t(s) = - K_{\zeta_t}(\theta_t(s))\,f_t(\theta_t(s)).
\end{align*}
The smooth family of reparametrized curves $\gamma_t:=\zeta_t\circ\theta_t$ satisfies
\begin{align*}
\partial_t\gamma_t
& = 
\frac{N_{\gamma_t}}{K_{\gamma_t}} + (f_t\circ\theta_t) T_{\gamma_t} + (\dot\zeta_t\circ\theta_t)\,\partial_t\theta_t\\
& =   \frac{N_{\gamma_t}}{K_{\gamma_t}} + \bigg(f_t\circ\theta_t  + \frac{\partial_t\theta_t}{K_{\gamma_t}} \bigg)T_{\zeta_t}
= 
\frac{N_{\gamma_t}}{K_{\gamma_t}}. \qedhere
\end{align*}
\end{proof}

In Theorem~\ref{t:nonradial}, we will need to deal with a given general convex compact set $C\subset\R^2$, which potentially may have empty interior or non-smooth boundary. Its associated support function
\begin{align*}
 h:\Su\to\R,
 \qquad
 h(s)=\max_{q\in C}\, \langle q,u(s)\rangle
\end{align*}
is continuous, and while it is not necessarily smooth, nevertheless it still satisfies the inequality $h+\ddot h\geq0$ in the sense of distributions. Conversely, any continuous function $h:\Su\to\R$ satisfying the latter inequality in the sense of distributions is the support function of a unique convex compact subset $C\subset\R^2$. The evolution of any such support function by means of the PDE \eqref{e:PDE_support} has good regularizing properties, as asserted by the following lemma. The existence part of the statement was also proved in \cite{Choi:2023aa}.

\begin{Lemma}\label{l:PDE_support_2}
For each continuous function $h_0:\Su\to\R$ as initial condition, the PDE \eqref{e:PDE_support} admits a unique continuous solution in the sense of distributions $h:[0,+\infty)\times \Su\to\R$, $h(t,s)=h_t(s)$, which is smooth on $(0,+\infty)\times\Su$ and has the form
\begin{align}
\label{e:form_of_ht}
 h_t(s) = e^t z + b(t,s),
\end{align}
where $b:[0,+\infty)\times \Su\to\R$ is uniformly bounded.
Moreover, if $h_0+\ddot h_0\neq0$ and $h_0+\ddot h_0\geq0$ in the sense of distributions, then 
\[h_t+\ddot h_t>0,\qquad\forall t>0.\]
\end{Lemma}

\begin{proof}
For each continuous function $h_0:\Su\to\R$ as initial condition, the family $h_t$, $t\geq0$, is a solution of~\eqref{e:PDE_support} in the sense of distributions if and only if the family $g_t:=e^{-t}h_t$ is a solution of the heat equation $\partial_tg_t=\ddot g_t$ in the sense of distributions. It is well known that the heat equation has the unique solution $g:[0,+\infty)\times \Su\to\R$, $g(t,s)=g_t(s)$, given by
\begin{align*}
 g_t(s)=h_0*k_t(s)=\int_{-\infty}^{+\infty} g_0(s-r)\,k_t(r)\,dr,
\end{align*}
where $k_t:\R\to\R$ is the heat kernel $k_t(r)=(4\pi t)^{-1/2}\exp(-r^2/4t)$. Therefore $g$ and likewise $h$ are everywhere continuous and smooth on $(0,+\infty)\times \Su$. If we express $h_0$ in Fourier series expansion as
\begin{align*}
 h_0(s)=z_0+\sum_{j\geq1} \langle z_j,u(js)\rangle,
\end{align*}
where $z_0\in\R$ and $z_j\in\R^2$ are the Fourier coefficients, we readily obtain that $g_t$ has the Fourier expansion
\begin{align*}
 g_t(s) = z_0 + \sum_{j\geq1} e^{-j^2t} \langle z_j,u(js)\rangle.
\end{align*}
This, together with the relation $h_t=e^tg_t$, shows that $h_t$ has the desired form~\eqref{e:form_of_ht}.

For $t>0$, we have $h_t=e^t h_0 \ast  k_t$.  Hence
\[
(h_t + \ddot h_t)(s)=e^t [(h_0 + \ddot h_0) \ast  k_t ] (s)= e^t \langle h_0 + \ddot h_0 , k_t (s -  \cdot) \rangle \, ,
\]
where $h_0 + \ddot h_0 \in {\mathcal S}'(\R)$ is seen as a $2\pi$-periodic tempered distribution (note that 
each function $k_t(s-\cdot)$ belongs to the Schwartz space ${\mathcal S} (\R)$). 
Now assume that  $h_0+\ddot h_0\neq0$ and $h_0+\ddot h_0\geq0$ in the sense of distributions. Since 
$k_t>0$, we conclude  that $(h_t + \ddot h_t)(s)>0$ for all $s\in \Su$.
\end{proof}

\begin{Remark}
The proof of Lemma~\ref{l:PDE_support_2} actually shows that the support functions $h_t$ have the form
\begin{align*}
 h_t(s) = e^t z_0 + \langle z_1,u(s)\rangle +e^{-3t}c(t,s),
\end{align*}
where $c:[0,+\infty)\times\Su\to\R$ is a continuous function whose restriction to any subset of the form $[t_0,\infty)\times\Su$, with $t_0>0$, is smooth and has bounded $C^k$ norm for every $k\geq0$. For $t>0$, the associated smooth solution $\gamma_t:\Su\hookrightarrow\R^2$ of the inverse curvature PDE \eqref{e:PDE} has the form
\begin{align*}
 \gamma_t(s)
 =
 h_t(s)u(s)+\dot h_t(s)\dot u(s)
 =
 z_1 + e^tz_0 u(s) + e^{-3t} f(t,s),
\end{align*}
where $f:=c\,u+\partial_s c\,\dot u$. In particular, this shows that $\gamma_t$ is asymptotic to a circle of center $z_1$ and radius $|e^tz_0|$ as $t\to\infty$.
\end{Remark}

\section{Levi potentials}\label{s:Levi_potentials}

\subsection{Level orbits}\label{ss:level_orbits}

Let us investigate the elementary properties of level orbits. Consider a smooth potential $U:\R^2\to\R$, the associated Newton equation $\ddot q=-\nabla U(q)$, and a level orbit $q$. We recall that $q$ is a solution of Hamilton equation defined on a maximal time interval and contained in some level set $U^{-1}(c)$.

\begin{Lemma}\label{l:const_speed}
Every level orbit $q(s)$ is defined for all $s\in\R$ and has constant speed, i.e.~$\tfrac d{ds}\|\dot q(s)\|\equiv0$. 
\end{Lemma}

\begin{proof}
The conservation of energy for the solutions of Hamilton equation implies that $\tfrac12\|\dot q(s)\|^2+U(q(s))$ is independent of $s$, and therefore the speed $v:=\|\dot q(s)\|\geq 0$ is independent of $s$ as well. This readily implies that the maximal domain of definition of the level orbit $q$ is the whole real line.
\end{proof}

From now on, we assume that $U$ is a Levi potential, and denote by $\reg(U):=\R^2\setminus\crit(U)$ the open subset of its regular points. For each $x\in\reg(U)$, we denote by 
\[\ell_x:=U^{-1}(U(x))\cap\reg(U)\] the regular part of the level set of $U$ containing $x$.
We define the vector field 
\[N:\reg(U)\to\R^2,\qquad N(x)=\frac{\nabla U(x)}{\|\nabla U(x)\|}.\]
Notice that $N(x)$ is a unit normal to $\ell_x$ pointing in the increasing direction of $U$. We orient each $\ell_x$ so that $J N(x)\in T_x\ell_x$ is a positive tangent vector, where $J:\R^2\to\R^2$, $J(x_1,x_2)=(-x_2,x_1)$ is the complex structure of $\R^2$. We say that a parametrized smooth immersed curve contained in some $\ell_x$ is \emph{direct} if its orientation agrees with the orientation of $\ell_x$.
Notice that Hamilton equation is reversible: if $q(t)$ is a solution, the backward curve $t\mapsto q(-t)$ is a solution as well. 
We define the smooth function
\begin{align*}
 K:\reg(U)\to\R
\end{align*}
such that $K(x)$ is the signed geodesic curvature of the oriented level set $\ell_x$ with respect to the normal $N(x)$. Namely, if $q$ is a direct level orbit such that $q(0)=x$, then
\begin{align*}
 K(x)=-\frac{\langle N(q(0)),\ddot q(0) \rangle}{\|\dot q(0)\|^2}.
\end{align*}

\begin{Lemma}\label{l:level_orbit_const_speed}
Any level orbit $q$ with $q(0)\in \reg(U)$ has constant speed 
\begin{align*}
\|\dot{q}(s)\|^2=\|\dot{q}(0)\|^2=\frac{\|\nabla U(q(0))\|}{K(q(0))},
\qquad\forall s\in\R.
\end{align*}
In particular, the curvature function $K$ is everywhere positive.
\end{Lemma}

\begin{proof}
Clearly, it is enough to argue for a direct level orbit $q$ with $q(0)\in \reg(U)$. Plugging Newton equation $\ddot q=-\nabla U(q)$ and the definition of $N$ into the expression of the curvature function $K$, we infer
\begin{align*}
 K(q(0))=-\frac{\langle  N(q(0)),\ddot q(0) \rangle}{\|\dot q(0)\|^2}=\frac{\|\nabla U(q(0))\|}{\|\dot q(0)\|^2}.
\end{align*}
This, together with the fact that the speed $\|\dot q\|$ is constant (Lemma~\ref{l:const_speed}), implies the lemma.
\end{proof}

We introduce the smooth function
\begin{align*}
v:\reg(U)\to(0,+\infty),\qquad v(x)=\sqrt{\frac{\|\nabla U(x)\|}{K(x)}},
\end{align*}
and the vector field
\begin{align*}
 V:\reg(U)\to\R^2,\qquad V(x)=v(x)\JJ N(x).
\end{align*}
Notice that the integral curves of $V$ are precisely the portions of direct level orbits $q$ in $\reg(U)$, and the function $v$ gives their speed, i.e.
\begin{align*}
 \dot q(s)=V(q(s)),\qquad\|\dot q(s)\|\equiv v(q(s)).
\end{align*}
In particular, the function $v$ is constant on every path-connected component of any level set $U^{-1}(c)\cap\reg(U)$.
Later on, in Lemma~\ref{l:level_orbits_in_reg}, we will show that the flow of $V$ is complete, that is, level orbits that intersect $\reg(U)$ are entirely contained in $\reg(U)$.

\subsection{Relations with the inverse curvature flow}
We define the vector field 
\begin{align*}
 W:\reg(U)\to\R^2,\qquad W(x)=\frac{N(x)}{K(x)},
\end{align*}
and denote by $\phi_t$ its flow. 

\begin{Prop}\label{p:level_orbit_inv_curv_flow}
Consider a smooth immersed curve $\gamma_0\subset \reg(U)$ contained in a level set of a Levi potential $U$. Then $\gamma_t:=\phi_t(\gamma_0)$ is also contained in a level set of $U$ for all $t\in\R$ for which it is well defined. In particular, if we fix a direct parametrization $\gamma_0:(a,b)\looparrowright\reg(U)$ and the corresponding parametrizations $\gamma_t:=\phi_t\circ\gamma_0:(a,b)\looparrowright\reg(U)$, then $\gamma_t$ is a solution of the inverse curvature flow.
\end{Prop}

\begin{proof}
For each $x\in\reg(U)$ and $t\in\R$ such that $\phi_t(x)$ is defined, we have
\begin{align}
\label{e:ODE_for_U_phit_x}
\tfrac{d}{dt} U(\phi_t(x)) 
&= 
\frac{\langle\nabla U(\phi_t(x)),N(\phi_t(x))\rangle}{K(\phi_t(x))} 
=
\frac{\|\nabla U(\phi_t(x))\|}{K(\phi_t(x))} 
=
v(\phi_t(x))^2.
\end{align}
Since the function $v$ is constant on the path-connected components of every level set $U^{-1}(c)\cap\reg(U)$, by the implicit function theorem there exists an open neighborhood $Z$ of any given point of $\reg(U)$ and a smooth function $f:U(Z)\to\R$ such that $f(U(y))=v(y)^2$ for all $y\in Z$. This, together with~\eqref{e:ODE_for_U_phit_x}, implies that, for all $t\in\R$ such that $\phi_t(x)\in Z$,  the function $t\mapsto U(\phi_t(x))$ is a solution of the ODE $\dot z(t)=f(z(t))$, and therefore is uniquely determined by the initial value $U(\phi_0(x))=U(x)$. Since the curve $\gamma_0$ is connected, this proves that $U(\phi_t(x))=U(\phi_t(y))$ for all $x,y\in\gamma_0$.

Finally, fix a direct parametrization $\gamma_0:(a,b)\looparrowright\reg(U)$, and the corresponding parametrizations $\gamma_t:=\phi_t\circ\gamma_0$. Since every $\gamma_t$ is contained in a level set of $U$, the vector $N(\gamma_t(s))$ is normal to $\dot\gamma_t(s)$, and the value $K(\gamma_t(s))$ is the curvature of  $\gamma_t$ at $\gamma_t(s)$ with respect to the orientations introduced in Section~\ref{ss:level_orbits}. Therefore, with the notation of Section~\ref{ss:PDE}, the vector field $W\circ\gamma_t$ coincides with $N_{\gamma_t}/K_{\gamma_t}$, and we conclude that $\gamma_t$ satisfies the inverse curvature flow PDE
\[
\partial_t\gamma_t = \frac{N_{\gamma_t}}{K_{\gamma_t}}.
\qedhere
\]
\end{proof}

\begin{Cor}\label{c:level_orbit_inv_curv_flow}
Let $q_0$ be a level orbit of a Levi potential $U$ with $q_0(0)\in\reg(U)$. We set $v_t:=v(\phi_t(q_0(0)))$ for all $t\in\R$ for which the right-hand side is defined. Then, on their maximal interval of definition containing $0$, the curves $q_t$ given by
\[q_t(s):=\phi_t(q_0(e^{-t}v_0^{-1}v_t s))\] 
are also portions of level orbits.
\end{Cor}

\begin{proof}
  We set $\gamma_t(s):=\phi_t(q_0(s))$, where $s$ varies in a maximal interval containing 0 for which the right-hand side is defined. Proposition~\ref{p:level_orbit_inv_curv_flow} implies that $\gamma_t$ is a solution of the inverse curvature flow, and each $\gamma_t$ is contained in a level set of $U$. Lemma~\ref{l:evolution_speed} implies that $\|\dot\gamma_t\|\equiv e^t \|\dot\gamma_0\|\equiv e^t v_0$. The reparametrized curve $q_t(s):=\gamma_t(e^{-t}v_0^{-1}v_t s)$ is an immersed curve contained in the level set of $U$ and has speed $\|\dot q_t\|\equiv\|\dot\gamma_t\| e^{-t}v_0^{-1}v_t \equiv v_t$. Therefore, it is a portion of a level orbit.
\end{proof}

\subsection{Level sets of a Levi potential}
As anticipated, we can now establish the completeness of the vector field $V$.

\begin{Lemma}\label{l:level_orbits_in_reg}
If a level orbit $q$ of a Levi potential $U$ satisfies $q(0)\in\reg(U)$, then $q(s)\in\reg(U)$ for all $s\in\R$.
\end{Lemma}

\begin{proof}
  We denote by $\psi_s$ the Hamiltonian (local) flow on $T^*\R^2=\R^2\times\R^2$ associated with the potential $U$. Namely, if $q$ is a solution of Hamilton equation defined on some maximal neighborhood of $0$, we have $(q(s),\dot q(s))=\psi_s(q(0),\dot q(0))$. 
We denote by $\pi:T^*\R^2\to\R^2$ the base projection of the cotangent bundle, and we define the family of maps
\begin{align*}
 \sigma_s:\reg(U)\to\R^2,\qquad\sigma_s(x)=\pi\circ\psi_s(x,V(x)).
\end{align*}
Notice that these maps are indeed well defined for all $s\in\R$. Indeed, $q(s):=\sigma_s(x)$ is the direct level orbit starting at $q(0)=x$, and Lemma~\ref{l:const_speed} guarantees that $q(s)$ is well defined for all $s\in\R$.

Assume by contradiction that $q$ is a level orbit of the Levi potential $U$ such that $q(0)\in\reg(U)$ and $q(s_0)\in\crit(U)$ for some $s_0\in\R\setminus\{0\}$. Without loss of generality, we assume that $q$ is a direct level orbit, and we consider the case $s_0>0$, the other case being analogous. By Proposition~\ref{p:level_orbit_inv_curv_flow}, the family of curves $\gamma_t(s):=\phi_t(q(s))$ is a solution of the inverse curvature flow on their maximal domain of definition. Moreover, if we set $w_t:=e^{t}v(\gamma_0(0))v(\gamma_t(0))^{-1}$ and $q_t(s):=\gamma_t(s/w_t)$, Corollary~\ref{c:level_orbit_inv_curv_flow} implies that every $q_t$ is a portion of a level orbit. For $t\in\R$ sufficiently close to 0, let $s_t\in(0,+\infty]$ be the largest positive value such that $\gamma_t|_{[0,s_t)}$ is contained in $\reg(U)$. 
Therefore 
\begin{align}
\label{e:gammas_psi_gamma0}
 \gamma_t(s)=\sigma_{w_t s}(\gamma_t(0)), \qquad\forall s\in[0,s_t).
\end{align} 
Since $\reg(U)$ is an open subset of $\R^2$, we have that $t\mapsto s_t$ is lower semicontinuous. Therefore, for each $s\in[0,s_0)$, we can differentiate the identity~\eqref{e:gammas_psi_gamma0} with respect to $t$ at $t=0$, and after taking the inner product with $N(q(s))$, we infer
\begin{equation}
\label{e:inv_curv_bound}
\begin{split}
 \frac{1}{K(q(s))} 
 &=
 \frac{\langle d\sigma_{w_t s}(q(0))N(q(0)),N(q(s))\rangle}{K(q(0))}
 +
 \underbrace{\langle V(q(s)),N(q(s))\rangle}_{=0} \tfrac{d}{dt} w_t s\\
 &\leq \frac{\|d\sigma_{w_t s}(q(0))\|}{K(q(0))}. 
\end{split}
\end{equation}
Since
\begin{align*}
 K(q(s))=-\frac{\langle  N(q(s)),\ddot q(s) \rangle}{\|\dot q(s)\|^2}=\frac{\|\nabla U(q(s))\|}{\|\dot q(0)\|^2}
\end{align*}
and $\nabla U(q(s_0))=0$, we infer that $K(q(s))\to0$ as $s\to s_0$. But this contradicts the uniform upper bound for $K(q(s))^{-1}$ given in~\eqref{e:inv_curv_bound}.
\end{proof}

We say that a level orbit $q$ of the Levi potential $U$ is \emph{regular} when it intersects $\reg(U)$, or equivalently when it is fully contained in $\reg(U)$ (Lemma~\ref{l:level_orbits_in_reg}). 

\begin{Lemma}\label{l:globally_defined}
Let $q$ be a regular level orbit of a Levi potential $U$ such that $\phi_t(q(0))$ is well defined for all $t\in[a,b]$, with $a<0<b$. Then $\phi_t(q(s))$ is well defined for all $t\in[a,b]$ and $s\in\R$.
\end{Lemma}

\begin{proof}
Assume without loss of generality that the regular level orbit $q$ is direct.
For each $t\in(a,b)$, since $\phi_t(q(0))\in\reg(U)$, Lemmas~\ref{l:const_speed} and~\ref{l:level_orbits_in_reg} imply that the regular direct level orbit $q_t$ such that $q_t(0)=\phi_t(q(0))$ is a curve of the form $q_t:\R\to\reg(U)$. We set $w_t:=e^tv(q(0))v(q_t(0))^{-1}$.  Corollary~\ref{c:level_orbit_inv_curv_flow} implies that $\phi_t(q(s))=q_t(w_ts)$ for all $s$ in a maximal neighborhood of 0 such that the left-hand side is defined.

Since $\phi_t(q(0))$ is in the open set $\reg(U)$ for all $t\in[a,b]$, there exists a maximal interval $(s_0,s_1)\subset\R$ such that $\phi_t(q(s))$ is a well defined point of $\reg(U)$ for all $t\in[a,b]$ and $s\in(s_0,s_1)$. We claim that $(s_0,s_1)=\R$. Indeed assume by contradiction that $s_1<+\infty$.  The curve $t\mapsto\phi_t(q(s_1))$ is defined on a neighborhood of $0$ but not on the whole $[a,b]$, and therefore exits every compact subset of $\reg(U)$. We readily obtain a contradiction: since $\phi_t(q(s))=q_t(w_ts)$ for all $s\in[0,s_1)$ and $t\in[a,b]$, the curve $t\mapsto\phi_t(q(s_1))$ is contained in the compact subset
\begin{align*}
\big\{q_t(w_ts)\ \big|\ t\in[a,b],\ s\in[0,s_1]\big\}\subset \reg(U).
\end{align*}
This proves that $s_1=+\infty$. Analogously, we have $s_0=-\infty$.
\end{proof}

\begin{Prop}\label{p:regular_level_orbit_periodic}
Any regular level orbit $q$ of a Levi potential $U$ is periodic, i.e. $q=q(\sigma+\cdot)$ for some $\sigma>0$.
\end{Prop}

\begin{proof}
For each $c\in\R$, the intersection $U^{-1}(c)\cap\reg(U)$ is a (possibly empty or disconnected) 1-dimensional properly embedded submanifold of $\reg(U)$, since it is a level set of the submersion $U|_{\reg(U)}$. If $q$ is a parametrization of a connected component of $U^{-1}(c)\cap\reg(U)$ with constant speed $\|\dot q\|\equiv v(q(0))$, then $q$ is a level orbit, and therefore $q(s)$ is defined and contained in $\reg(U)$ for all $s\in\R$ (Lemma~\ref{l:level_orbits_in_reg}). This shows that the regular level orbits are parametrizations of the connected components of $U^{-1}(c)\cap\reg(U)$.
Assume by contradiction that there exists a regular level orbit $q$ that is not periodic. 

Since $q$ is a 1-dimensional properly embedded submanifold in $\reg(U)$ and is not a circle, as a map it is a proper embedding $q:\R\hookrightarrow \reg(U)$. We claim that $q$ is also proper as a map $q:\R\hookrightarrow \R^2$. If this is not true, then by Lemma~\ref{l:curvature_plane_lines} at least one of the half orbits $q((-\infty,0])$ or $q([0,+\infty))$ is contained in a compact subset $C\subset\R^2$. Consider the case $q([0,+\infty))\subset C$, the other one being analogous. Since $q:[0,+\infty)\hookrightarrow \reg(U)$ is proper, its $\omega$-limit is contained in the compact subset $C\cap\crit(U)$. Therefore 
\begin{align*}
\epsilon_s
:= 
\sup_{r\geq s}\|\ddot q(r)\|
=
\sup_{r\geq s}\|\nabla U(q(r))\|\ttoup_{s\to+\infty}0.
\end{align*}
This implies, for each $s>0$ and $r:=s+\tfrac{v(q(0))}{2\epsilon_s}$,
\begin{align*}
 \|q(r)-q(s)\|
 &=
 \left\|(r-s)\dot q(s) + \int_s^{r}\big( \dot q(u) - \dot q(s) \big)\,du\right\|\\
 &\geq 
 (r-s)v(q(0)) - \int_{s}^r (u-s)\epsilon_s\,du\\
 &\geq
 \frac{v(q(0))^2}{4\epsilon_s}
 \ttoup_{s\to+\infty} +\infty,
\end{align*}
contradicting the fact that $q|_{[0,+\infty)}$ is contained in the compact set $C$.

  We showed that $q:\R\hookrightarrow\R^2$ is a proper embedding. Let $\tau>0$ be small enough so that $\phi_t(q(0))$ is well defined for all $t\in[0,\tau]$. Lemma~\ref{l:globally_defined} implies that $\gamma_t(s):=\phi_t(q(s))$ is well defined for all $t\in[0,\tau]$ and $s\in\R$. Corollary~\ref{c:level_orbit_inv_curv_flow} implies that each $\gamma_t$ is a reparametrization of a regular level orbit. But Proposition~\ref{p:level_orbit_inv_curv_flow} implies that $\gamma_t$ is a solution of the inverse curvature flow defined for $t\in[0,\tau]$, which violates Proposition~\ref{p:open}.\end{proof}

\begin{proof}[Proof of Theorem~\ref{t:smooth}]
Let $U$ be a smooth Levi potential whose set of critical points $\crit(U)$ is totally path-disconnected, and $q$ a regular direct level orbit. By Proposition~\ref{p:regular_level_orbit_periodic}, $q$ is $\sigma$-periodic for some $\sigma>0$. Therefore, $q$ is an embedding of the form $q:\R/\sigma\Z\hookrightarrow\reg(U)$ with positive curvature. We denote by $\tau\in[-\infty,0)$ the infimum of the values $t<0$ such that $\phi_t(q(0))$ is a well defined point of $\reg(U)$. Lemma~\ref{l:globally_defined} implies that the curves $\gamma_t:=\phi_t\circ q:\R/\sigma\Z\hookrightarrow\reg(U)$ 
are well defined for all $t\in(\tau,0]$. 
Proposition~\ref{p:level_orbit_inv_curv_flow} implies that the family $\gamma_t$ is a solution of the inverse curvature flow, and Corollary~\ref{c:level_orbit_inv_curv_flow} implies that each $\gamma_t$ is a reparametrization of a regular level orbit, and actually a direct one (since $\gamma_0=q$ is direct). In particular each $\gamma_t$ has positive curvature, and therefore it bounds a convex compact subset $C_t\subset\R^n$. The inverse curvature PDE~\eqref{e:PDE} readily implies that $C_{t_1}\subset\interior(C_{t_2})$ for all $t_1,t_2\in(\tau,0]$ with $t_1<t_2$.

We claim that $\tau=-\infty$. Indeed, let us assume by contradiction that $\tau$ is finite. This implies that, for every $s\in\R/\sigma\Z$, the curve 
$(\tau,0]\ni t\mapsto\gamma_t(s)$ exits any compact subset of $C_0\cap\reg(U)$. Therefore, the non-empty convex compact set
\begin{align*}
 C:=\bigcap_{t\in(-\tau,0]} C_t
\end{align*}
has boundary contained in $\crit(U)$. Since $\partial C$ is path-connected (being the boundary of a convex compact set) while $\crit(U)$ is totally path-disconnected, $\partial C$ must be a singleton. However, this implies that 
\begin{align}\label{e:limit_length_ThmA}
\ell_t:=\mathrm{length}(\gamma_t)\ttoup_{t\to\tau^+}0.
\end{align} 
Since the regular level orbit $q$ is parametrized with constant speed $\|\dot q\|\equiv v(q(0))$, each $\gamma_t$ is parametrized with constant speed $\|\dot\gamma_t\|\equiv e^t v(q(0))$ (Lemma~\ref{l:evolution_speed}), and therefore $\ell_t=e^t v(q(0))\sigma \to e^{\tau} v(q(0))\sigma$ as $t\to\tau^+$, which contradicts~\eqref{e:limit_length_ThmA}.

We proved that the family of smooth periodic curves $\gamma_t:\R/\sigma\Z\hookrightarrow\reg(U)$, for $t\in(-\infty,0]$, is a solution of the inverse curvature flow. Proposition~\ref{p:periodic} implies that $\gamma_t$ is a circular solution (Example~\ref{ex:circular_solution}). Let $\tau>0$ be the supremum of the time values $t>0$ such that $\phi_t(q(0))$ is well defined. 
Using Lemma~\ref{l:globally_defined}, Proposition~\ref{p:level_orbit_inv_curv_flow}, and Corollary~\ref{c:level_orbit_inv_curv_flow} as before, we can extend the family of curves $\gamma_t:=\phi_t\circ q:\R/\sigma\Z\hookrightarrow\reg(U)$ for all $t\in(-\infty,\tau)$, such a family is a solution of the inverse curvature flow, and each $\gamma_t$ is a reparametrization of a regular level orbit. Therefore, the first part of the proof implies that $\gamma_t$ is a circular solution of the inverse curvature flow, i.e.
\begin{align*}
 \gamma_t(s):=x_0 + e^t e^{2\pi s/\sigma  \JJ}x_1,\qquad\forall t\in(-\infty,\tau),\ s\in\R/\sigma\Z,
\end{align*}
where $x_0\in\R^2$, $x_1\in\R^2\setminus\{0\}$.
We claim that $\tau=+\infty$. Indeed, if $\tau<+\infty$, then for each $s\in\R/\sigma\Z$ the curve $[0,\tau)\ni t\mapsto\gamma_t(s)$ exits any compact subset of $\reg(U)$; therefore the curves $\gamma_t$ would converge as $t\to\tau^-$ to an embedded circle in $\crit(U)$, contradicting the fact that $\crit(U)$ is totally path-disconnected.

Summing up, we proved that every round circle centered at $x_0\in\R^2$ is a regular level set of $U$. Therefore $U$ has a unique critical point at $x_0$, and can be written as $U(x)=f(\|x-x_0\|^2)$ for some smooth function $f:[0,+\infty)\to\R$. 
\end{proof}

\subsection{Analytic Levi potential}
In order to prove Theorem~\ref{t:analytic}, we first need some preliminaries on analytic functions. We begin by recalling the real version of the classical Puiseux theorem from, e.g.,
\cite[page104]{Ghys:2017aa}, and some of its consequences.

\begin{Thm}[Puiseux]
\label{t:Puiseux}
Let $U:\R^2\to\R$, $(x,y)\mapsto U(x,y)$ be a non-constant analytic function that vanishes at the origin and is not divisible by $x$ $($which can always be achieved by means of an analytic change of variables$)$. Then, in a neighborhood of the origin, the level set $U^{-1}(0)$ is either equal to the origin, or is the union of a finite number of arcs of the form
\[\gamma_i:[-\epsilon_i,\epsilon_i]\to\R^2,\qquad\gamma_i(t)=((-1)^{k_i}t^{m_i},f_i(t)),\]
for some positive integers $k_i\in\{0,1\}$ and $m_i\geq1$, and for some analytic functions $f_i$ such that
$f_i(0)=0$. The arcs $\gamma_i$ are injective, and their images only intersect at the origin.
\end{Thm}

\begin{Cor}\label{c:crit_locally_path_connected}
For each analytic function $U:\R^2\to\R$, the set of critical points $\crit(U)$ is locally path connected.
\end{Cor}

\begin{proof}
The function $V:=\|\nabla U\|^2$ is also analytic, and $\crit(U)=V^{-1}(0)$. The level set $V^{-1}(0)$ is locally path connected according to Theorem~\ref{t:Puiseux}.
\end{proof}

We recall that a planar graph is a graph topologically embedded in $\R^2$. Here, the graph is endowed with the usual topology that makes it a CW complex, with the vertices being the 0-cells and the edges being the 1-cells. The degree of a vertex $v$ of a graph is the number of edges incident to $v$, where an edge incident to $v$ at both ends is counted twice.

\begin{Cor}\label{c:planar_graph}
Let $U:\R^2\to\R$ be an analytic function. Each compact connected component of a level set $U^{-1}(c)$ is a planar graph $($possibly empty or with no edges$)$ whose vertices have even degrees.
\end{Cor}

\begin{proof}
Let us assume that $U^{-1}(c)$ is not empty, and $\Gamma\subset U^{-1}(c)$ is a compact connected component that is not a singleton.  
By Theorem~\ref{t:Puiseux}, each point $q\in\Gamma$ has a bounded open neighborhood $B_q\subset\R^2$ such that $U^{-1}(c)\cap B_q$ is the union of finitely many arcs $\gamma_{q,1},...,\gamma_{q,n_q}$ with endpoints in $\partial B_q$, where each arc is without self-intersections and distinct arcs intersect only at $q$. Since $\Gamma$ is compact, there exist finitely many points $q_1,...,q_k$ such that $B_{q_1}\cup...\cup B_{q_k}$ contains $\Gamma$. We can now endow $\Gamma$ with a planar graph structure, whose vertices are $q_1,...,q_k$, and whose edges are the finitely many connected components of $\Gamma\setminus\{q_1,...,q_k\}$. Notice that every arc $\gamma_{q_i,j}$ intersects precisely two edges incident to the vertex $q_i$, or a single edge incident to the vertex $q_i$ at both ends. This, together with the fact that $\gamma_{q_i,j}\cap\gamma_{q_i,h}=\{q_i\}$ for all $j\neq h$, readily implies that the degree of each vertex $q_i$ is even.
\end{proof}

\begin{Cor}\label{c:local_extremum}
Let $U:\R^2\to\R$ be an analytic function such that the complement of a level set $\R^2\setminus U^{-1}(c)$ has a non-empty bounded connected component. Then $U$ has a strict local maximum or a strict local minimum.
\end{Cor}

\begin{proof}
  We assume by contradiction that $U$ does not have a strict local maximum nor a strict local minimum. By assumption, $\R^2\setminus U^{-1}(c)$ has a non-empty bounded connected component $V$.   
The complement $\R^2 \setminus V$ has a unique unbounded connected component $W$.
We define $A_1:=\R^2\setminus W$, which is a bounded connected open set containing $V$.
We recall that a connected open subset of $\R^2$ is simply connected if and only if its complement has no bounded connected component (see, e.g., \cite[Corollaries~1-2]{Basye:1935aa}).  Therefore $A_1$ is simply connected. Moreover the boundary of $A_1$ is contained in $U^{-1}(c)$.
Let us assume that $U|_{V}<c$, the case in which $U|_{V}>c$ being analogous. We set $a_1:=\min U|_{A_1}$. The intersection $U^{-1}(a_1)\cap A_1$ is compact, and thus it is a union of finitely many connected components of the level set $U^{-1}(a_1)$. Notice that $U^{-1}(a_1)$ does not contain isolated points, for otherwise any such point would be a strict local minimum of $U$. 
Corollary~\ref{c:planar_graph} implies that $U^{-1}(a_1)\cap A_1$ is a planar graph whose vertices have non-zero even degrees. In particular, $U^{-1}(a_1)\cap A_1$ contains a loop that bounds a simply connected component $B_1$ of $A_1 \setminus U^{-1}(a_1)$.

We set $b_1:=\max U|_{B_1}$. Arguing as in the previous paragraph, the intersection $U^{-1}(b_1)\cap B_1$ is a compact planar graph, and contains a loop that bounds a simply connected component $A_2$ of $B_1\setminus U^{-1}(b_1)$. Next we define $a_2:=\min U|_{A_2}$, and continue the process iteratively. Overall, we obtained a sequence of simply connected non-empty open sets $A_i,B_i\subset \R^2$ such that
\begin{align*}
 \overline {B_i}\subset A_i,
 \quad
 \overline {A_{i+1}}\subset B_i,
 \quad
 \partial B_i\subset \crit(U)\cap U^{-1}(a_i),
 \quad
 \partial A_{i+1}\subset \crit(U)\cap U^{-1}(b_i),
\end{align*}
where $a_i<b_i$. By Sard theorem, the interval $(a_i,b_i)$ contains a full measure subset of regular values of $U$. Therefore, the boundaries $\partial B_i$ and $\partial A_{i+1}$ belong to distinct path-connected components of $\crit(U)$. This further implies that, for each $i< j$, the boundaries $\partial B_i$ and $\partial B_j$ belong to distinct path-connected components of $\crit(U)$, since they belong to distinct path-connected components of the complement of $\partial A_{i+1}$.
Consider an arbitrary sequence $q_i\in\partial B_i$, which is contained in the bounded set $B_1$. Up to extracting a subsequence, we have $q_i\to q\in\crit(U)$ as $i\to\infty$. Since all the points $q_i$ belong to pairwise distinct path-connected components of $\crit(U)$, we infer that $\crit(U)$ is not locally path-connected at $q$. This contradicts Corollary~\ref{c:crit_locally_path_connected}.
\end{proof}

\begin{proof}[Proof of Theorem~\ref{t:analytic}]
Let $U:\R^2\to\R$ be an analytic Levi potential. If $U$ is a constant function, it is also trivially radial. Assume now that $U$ is not constant, and let $c\in\R$ be a regular value of $U$ such that $U^{-1}(c)$ is not empty. By Proposition~\ref{p:regular_level_orbit_periodic}, the connected components of the level set $U^{-1}(c)$ are circles. Therefore, $\R^2\setminus U^{-1}(c)$ has a bounded connected component, and Corollary~\ref{c:local_extremum} implies that $U$ has a strict local minimum or a strict local maximum.

We claim that $U$ cannot have a strict local maximum. Indeed, assume that $q\in\R^2$ is a strict local maximum of $U$, and set $c:=U(q)$. Let $W\subset\R^2$ be a sufficiently small compact neighborhood of $q$ such that $U|_{W\setminus\{q\}}<c$. By Corollary~\ref{c:crit_locally_path_connected}, $\crit(U)$ is locally path-connected, and therefore $\crit(U)\cap W=\{q\}$ provided $W$ is chosen small enough. For each $\epsilon>0$ small enough, the level set $U^{-1}(c-\epsilon)$ has a connected component $\gamma$ contained in $W$. Since $c-\epsilon$ is a regular value of $U|_W$, $\gamma$ is an embedded circle in $W$. Notice that $\gamma$ bounds a disk $B$ that must contain a local maximum or a local minimum of $U$. Since the only critical point of $U|_W$ is $q$, we infer that $\crit(U)\cap B=\{q\}$. Since $q$ is a local maximum of $U$, we infer that, for each $q'\in\gamma$, the gradient $\nabla U(q')$ points inside $B$. Lemma~\ref{l:level_orbit_const_speed} implies that the curvature of $\gamma$ with respect to the normal vector field $\nabla U/\|\nabla U\|$ is everywhere positive, but this is impossible, as it would prevent $\gamma$ to encircle $q$.

Since $U$ cannot have strict local maxima, it must have a strict local minimum $q_0$. Without loss of generality, let us assume that $q_0=0$ and $U(q_0)=0$. With the same argument already employed in the last paragraph, There exists $\epsilon>0$ such that, for each $c\in(0,\epsilon]$, the level set $U^{-1}(c)$ has a connected component  that is a circle, does not contain critical points of $U$, and encircles the origin. Let $\gamma_0\subset U^{-1}(0,\epsilon]$ be any such circle. Since $U$ is a Levi potential, $\gamma_0$ must be a regular level orbit. We now proceed as in the proof of Theorem~\ref{t:smooth}: we apply to $\gamma_0$ the inverse curvature flow~\eqref{e:PDE} in negative time, and obtain family of curves $\gamma_t$ defined for $t$ in some neighborhood of 0 in $(-\infty,0]$. Proposition~\ref{p:level_orbit_inv_curv_flow} implies that each $\gamma_t$ is a regular level orbit for the Levi potential, and Lemma~\ref{l:const_speed} implies $\gamma_t$ is the boundary of a convex open subset $B_t$ containing the origin. Arguing as in the proof of Theorem~\ref{t:smooth}, we infer that $\gamma_t$ is defined for all $t\leq0$. Proposition~\ref{p:periodic} implies that $\gamma_t$, for $t\leq0$, is a circular solution of the inverse curvature flow. This proves that $U$ is radial in $B_0$. The radial condition can be expressed by saying that the function $V(q):=dU(q)Jq$ vanishes on $B_0$. However, $V$ is analytic, and since it vanishes in the open set $B_0$ it must vanish identically on the whole $\R^2$. This proves that $U$ is a radial function.
\end{proof}

\subsection{Levi potentials with prescribed critical set}
The proof of Theorem~\ref{t:nonradial} is a consequence of Lemmas~\ref{l:PDE_support}, \ref{l:PDE_support_2}, and of the following statement.

\begin{Lemma}\label{l:smoothing}
Let $f:\R^n\to[0,+\infty)$ be a proper continuous function such that $\min f=0$ and whose restriction to $f^{-1}(0,+\infty)$ is smooth and has no critical points. There exists a smooth function $g:\R^n\to[0,+\infty)$ with critical set $\crit(g)=g^{-1}(0)=f^{-1}(0)$ and having the same level sets as $f$, i.e.~$f(x)=f(y)$ if and only if $g(x)=g(y)$ for all $x,y\in\R^n$.
\end{Lemma}

\begin{proof}
Let $\psi:[0,+\infty)\to[0,+\infty)$ be a smooth function such that $\psi|_{[0,1]}\equiv0$, $\dot\psi|_{(1,2]}>0$, and $\dot\psi|_{[2,+\infty)}\equiv1$. For each integer $k\geq1$, we define
\begin{align*}
 f_k:\R^n\to[0,+\infty),\qquad f_k(x)=\psi(kf(x)).
\end{align*}
We introduce the compact subsets 
\[W_k:=f^{-1}[0,\tfrac1k].\] 
Notice that $f$ is smooth outside $f^{-1}(0)$,  and $f_k|_{W_k}\equiv0$. Therefore $f_k$ is everywhere smooth. We set
$c_k:=1+\|f_k|_{W_1}\|_{C^k}$, and define the function
\begin{align*}
g:\R^n\to[0,+\infty),
\qquad
g(x)
=
\sum_{k\geq1}
\frac{f_k(x)}{e^kc_k}.
\end{align*}
Since $f$ is smooth away from $f^{-1}(0)$, on every compact subset $V\subset\R^n\setminus f^{-1}(0)$ and for every integer $p\geq1$, we have
\begin{align*}
 \big\|f_k|_{V}\big\|_{C^p} 
 \leq 
 b\, k^p,
\end{align*}
where $b=b(\psi,f,p,V)>0$ is independent of the integer $k$.
This estimate implies that $g|_{V}$ is smooth. Moreover,
\begin{align*}
\big\|f_k|_{W_1}\big\|_{C^p} 
\leq
\max
\Big\{ \big\|f_1|_{W_1}\big\|_{C^p},\big\|f_2|_{W_1}\big\|_{C^p},...,\big\|f_p|_{W_1}\big\|_{C^p}, c_k \Big\}.
\end{align*}
This shows that $g|_{W_1}$ is smooth as well, and therefore $g$ is everywhere smooth. Notice that $g^{-1}(0)=f^{-1}(0)$. Moreover, for each $x\in\R^n\setminus f^{-1}(0)$, we have $df(x)\neq0$, and therefore
\begin{align*}
dg(x)=\underbrace{\sum_{k\geq1} \frac{k\,\dot\psi(k f(x))}{e^kc_k}}_{>0}df(x)\neq0. 
\end{align*}
This implies that $\crit(g)=g^{-1}(0)=f^{-1}(0)$. Finally, since $g$ is of the form $h\circ f$ for a strictly monotone increasing function $h:[0,+\infty)\to[0,+\infty)$, we conclude that $f$ and $g$ have the same level sets.
\end{proof}

\begin{proof}[Proof of Theorem~\ref{t:nonradial}]
Let $C$ be a non-empty compact convex subset of $\R^2$. If $C=\{q_0\}$, then $U(q)=\|q-q_0\|^2$ is a Levi potential with critical set $\crit(U)=\{q_0\}$. Assume now that $C$ is not a singleton, so that its support function
\begin{align*}
 h_0:\Su\to\R,\qquad h_0(s)=\max_{q\in C} \langle q,u(s) \rangle,
\end{align*}
satisfies $h_0+\ddot h_0\geq0$ and $h_0+\ddot h_0\neq0$ in the sense of distributions. Here, as in Section~\ref{ss:PDE_support}, $u(s)=(\cos(s),\sin(s))$. By Lemma~\ref{l:PDE_support_2}, using the support function $h_0$ as initial condition, there exists a unique continuous solution $h:[0,+\infty)\times\Su\to\R^2$, $h(t,s)=h_t(s)$ of the PDE $\partial_t h_t= h_t+\ddot h_t$, such that:
\begin{itemize}
\item[(i)] $h$ is smooth and satisfies $h_t+\ddot h_t>0$ on $(0,+\infty)\times\Su$,

\item[(ii)] $h$ has the form $h_t(s) = e^t z + b(t,s)$, 
where $b:[0,+\infty)\times\Su\to\R^2$ is uniformly bounded.
\end{itemize}
Point (i) implies that, for each $t>0$, $h_t$ is the support function of a compact convex subset $C_t\subset\R^2$ with smooth positively-curved boundary. Moreover, since $\partial_th_t>0$ for all $t>0$, we have $C_{t_1}\subset \interior(C_{t_2})$ for all $t_1<t_2$. Since $\lim_{t \to 0+} h_t(s)=h_0(s)$ for all $s\in \Su$,
we have
\begin{align*}
 C_0=\bigcap_{t>0} C_t.
\end{align*}
Point (ii) implies that $\min h_t\to+\infty$ as $t\to+\infty$, and therefore the family $C_t$, $t>0$, is an exhaustion by compact sets of $\R^2$, i.e.
\begin{align*}
 \bigcup_{t>0} C_t = \R^2.
\end{align*}
By Lemma~\ref{l:PDE_support}, there exists a family of smooth periodic curves $\gamma_t:\Su\to\R^2$, $t>0$, that evolve according to the inverse curvature flow, and each $\gamma_t$ is a parametrization of the boundary of $C_t$. Without loss of generality, we can assume that each $\gamma_t$ is parametrized with constant speed $\|\dot\gamma_t\|$.
Notice that this family, seen as a map 
\[\gamma:(0,+\infty)\times\Su\to\R^2\setminus C_0,\qquad\gamma(t,s)=\gamma_t(s),\]
is a diffeomorphism. We define the continuous function $\tau:\R^2\to[0,+\infty)$ by $\tau|_{C_0}\equiv0$ and $x\in\gamma_{\tau(x)}(\Su)$ for all $x\in\R^2\setminus C_0$. Notice that $\tau$ is proper. Moreover,
the restriction $\tau|_{\R^2\setminus C_0}$ is smooth, strictly positive, and has no critical points. By Lemma~\ref{l:smoothing}, there exists a smooth function $U:\R^2\to[0,+\infty)$ with critical set $\crit(U)=U^{-1}(0)=\tau^{-1}(0)=C_0$ and having the same level sets as $\tau$. Namely, the level sets of $U$ are $C_0$ and the curves $\gamma_t$. We set $w(t):=U(\gamma_t(s))$, and stress that $w(t)$ is independent of $s\in\Su$ and smooth for all $t>0$. Moreover, since $\nabla U(\gamma_t(s))$ is positively proportional to the normal vector $N_{\gamma_t}(s)$, we have 
\begin{align*}
 \dot w(t)
 =
 \langle\nabla U(\gamma_t(s)),\partial_t\gamma_t(s)\rangle
 =
 \frac{\|\nabla U(\gamma_t(s))\|}{K_{\gamma_t}(s)}>0.
\end{align*}
The family of reparametrized curves
\[q_t(s):=\gamma_t\Big(\tfrac{\sqrt{\dot w(t)}}{\|\dot\gamma_t\|}s\Big)\]
are level orbits of $U$. Indeed, for $r=\tfrac{\sqrt{\dot w(t)}}{\|\dot\gamma_t\|}s$, we have
\begin{align*}
 \ddot q_t(s) 
 &= 
 \tfrac{\dot w(t)}{\|\dot\gamma_t\|^2}\ddot\gamma_t(r)
 =
 -\tfrac{\dot w(t)}{\|\dot\gamma_t\|^2} \|\dot\gamma_t\|^2 K_{\gamma_t}(r) N_{\gamma_t}(r)\\
 &=
 -\|\nabla U(q_t(s))\| N_{q_t}(s)
 =
 -\nabla U(q_t(s)).
\end{align*}
This shows that the collection of all level orbits $q_t$, for $t>0$, fills the regular set $\reg(U)$. All the points of critical set $\crit(U)$ are trivially level orbits. Therefore, $U$ is a Levi potential.
\end{proof}

\bibliography{_biblio}

\end{document}

%% file: incident.pdf_tex
\begingroup%
  \makeatletter%
  \providecommand\color[2][]{%
    \errmessage{(Inkscape) Color is used for the text in Inkscape, but the package 'color.sty' is not loaded}%
    \renewcommand\color[2][]{}%
  }%
  \providecommand\transparent[1]{%
    \errmessage{(Inkscape) Transparency is used (non-zero) for the text in Inkscape, but the package 'transparent.sty' is not loaded}%
    \renewcommand\transparent[1]{}%
  }%
  \providecommand\rotatebox[2]{#2}%
  \newcommand*\fsize{\dimexpr\f@size pt\relax}%
  \newcommand*\lineheight[1]{\fontsize{\fsize}{#1\fsize}\selectfont}%
  \ifx\svgwidth\undefined%
    \setlength{\unitlength}{381.33769839bp}%
    \ifx\svgscale\undefined%
      \relax%
    \else%
      \setlength{\unitlength}{\unitlength * \real{\svgscale}}%
    \fi%
  \else%
    \setlength{\unitlength}{\svgwidth}%
  \fi%
  \global\let\svgwidth\undefined%
  \global\let\svgscale\undefined%
  \makeatother%
  \begin{picture}(1,0.29038012)%
    \lineheight{1}%
    \setlength\tabcolsep{0pt}%
    \put(0,0){\includegraphics[width=\unitlength,page=1]{incident.pdf}}%
    \put(0.22804469,0.03877754){\color[rgb]{0,0,0}\makebox(0,0)[lt]{\lineheight{1.25}\smash{\begin{tabular}[t]{l}$\gamma$\end{tabular}}}}%
    \put(0.35982706,0.14551678){\color[rgb]{0,0,0}\makebox(0,0)[lt]{\lineheight{1.25}\smash{\begin{tabular}[t]{l}$\gamma(s_2)$\end{tabular}}}}%
    \put(0.19292267,0.27746508){\color[rgb]{0,0,0}\makebox(0,0)[lt]{\lineheight{1.25}\smash{\begin{tabular}[t]{l}$\dgray{\ell_2}$\end{tabular}}}}%
    \put(0.32742497,0.27083357){\color[rgb]{0,0,0}\makebox(0,0)[lt]{\lineheight{1.25}\smash{\begin{tabular}[t]{l}$\dgray{\ell_1}$\end{tabular}}}}%
    \put(0.01053038,0.14551678){\color[rgb]{0,0,0}\makebox(0,0)[lt]{\lineheight{1.25}\smash{\begin{tabular}[t]{l}$ $\end{tabular}}}}%
    \put(0.21685452,0.11665157){\color[rgb]{0,0,0}\makebox(0,0)[lt]{\lineheight{1.25}\smash{\begin{tabular}[t]{l}$C$\end{tabular}}}}%
    \put(0.10395056,0.17512594){\color[rgb]{0,0,0}\makebox(0,0)[lt]{\lineheight{1.25}\smash{\begin{tabular}[t]{l}$\gamma(s_1)$\end{tabular}}}}%
    \put(0,0){\includegraphics[width=\unitlength,page=2]{incident.pdf}}%
    \put(0.71580182,0.03877754){\color[rgb]{0,0,0}\makebox(0,0)[lt]{\lineheight{1.25}\smash{\begin{tabular}[t]{l}$\gamma$\end{tabular}}}}%
    \put(0.8475843,0.14551678){\color[rgb]{0,0,0}\makebox(0,0)[lt]{\lineheight{1.25}\smash{\begin{tabular}[t]{l}$\gamma(s_2)$\end{tabular}}}}%
    \put(0.68067997,0.27746508){\color[rgb]{0,0,0}\makebox(0,0)[lt]{\lineheight{1.25}\smash{\begin{tabular}[t]{l}$\dgray{\ell_2}$\end{tabular}}}}%
    \put(0.81518204,0.27083357){\color[rgb]{0,0,0}\makebox(0,0)[lt]{\lineheight{1.25}\smash{\begin{tabular}[t]{l}$\dgray{\ell_1}$\end{tabular}}}}%
    \put(0.59170752,0.17512594){\color[rgb]{0,0,0}\makebox(0,0)[lt]{\lineheight{1.25}\smash{\begin{tabular}[t]{l}$\gamma(s_1)$\end{tabular}}}}%
    \put(0,0){\includegraphics[width=\unitlength,page=3]{incident.pdf}}%
    \put(0.8020872,0.18278177){\color[rgb]{0,0,0}\makebox(0,0)[lt]{\lineheight{1.25}\smash{\begin{tabular}[t]{l}$\gamma(s_0)$\end{tabular}}}}%
    \put(0.22804469,0.00337632){\color[rgb]{0,0,0}\makebox(0,0)[lt]{\lineheight{1.25}\smash{\begin{tabular}[t]{l}$\textbf{(a)}$\end{tabular}}}}%
    \put(0.71580182,0.00337632){\color[rgb]{0,0,0}\makebox(0,0)[lt]{\lineheight{1.25}\smash{\begin{tabular}[t]{l}$\textbf{(b)}$\end{tabular}}}}%
  \end{picture}%
\endgroup%